\documentclass[11pt]{amsart}
\usepackage{amssymb,amsmath,amsthm}
\usepackage[shortlabels]{enumitem}

\theoremstyle{plain}
\newtheorem{theorem}{Theorem}[section]
\newtheorem{corollary}[theorem]{Corollary}

\theoremstyle{remark}
\newtheorem*{remark}{Remark}

\newcommand{\CC}{{\mathbb C}}

\renewcommand{\Re}{\operatorname{Re}}

\begin{document}

\title{Contractivity of M\"obius functions of operators}

\date{}

\author[T. Ransford]{Thomas Ransford}
\address{D\'epartement de math\'ematiques et de statistique, Universit\'e Laval,
Qu\'ebec City (Qu\'ebec),  G1V 0A6, Canada}
\email[Corresponding author]{ransford@mat.ulaval.ca}

\author[D. Tsedenbayar]{Dashdondog Tsedenbayar}
\address{Department of  Mathematics, Mongolian University of Science and Technology,
P.O.\ Box 46/520, Ulaanbaatar, Mongolia}
\email{cdnbr@yahoo.com}

\begin{abstract}
Let $T$ be a  injective bounded linear operator on a complex Hilbert space.
We characterize the complex numbers $\lambda,\mu$
for which $(I+\lambda T)(I+\mu T)^{-1}$ is a contraction, 
the characterization being expressed in terms of the numerical range of the possibly unbounded operator $T^{-1}$.

When $T=V$, the Volterra operator on $L^2[0,1]$,
this leads  to a
result of Khadkhuu, Zem\'anek and the second author, characterizing those $\lambda,\mu$ for which 
$(I+\lambda V)(I+\mu V)^{-1}$ is a contraction.
Taking $T=V^n$,
we further deduce that  $(I+\lambda V^n)(I+\mu V^n)^{-1}$
is never a contraction if $n\ge2$ and $\lambda\ne\mu$.
\end{abstract}


\makeatletter
\@namedef{subjclassname@2020}{\textup{2020} Mathematics Subject Classification}
\makeatother


\maketitle


\section{Introduction and statement of results}\label{S:intro}

Let $H$ be a complex Hilbert space,
and let $T$ be a bounded linear operator on $H$.
It is well known that  $\|e^{-tT}\|\le1$ for all $t>0$
if and only if the numerical range of $T$ 
 is contained in the right half-plane $\{z\in\CC:\Re z\ge0\}$.
This is a version of the celebrated Lumer--Phillips theorem \cite{LP61}.

In this note, we study  of operators of the form
$(I+\lambda T)(I+\mu T)^{-1}$, where $\lambda,\mu\in\CC$.
Our goal is to obtain a necessary and sufficient condition for when
$\|(I+\lambda T)(I+\mu T)^{-1}\|\le1$. In contrast with the 
Lumer--Phillips theorem, our condition is expressed
terms of the numerical range of $T^{-1}$ rather than that of $T$.
Here we assume that $T$ is injective, so $T^{-1}$ well-defined on $T(H)$,
though it may well be an unbounded operator.
We should therefore make precise the notion of numerical range
of an unbounded operator.

If $A$ is a (possibly unbounded) operator defined on a domain $D(A)\subset H$,
then the numerical range of $A$ is defined by
\[
W(A):=\{\langle Ax,x\rangle: x\in D(A), \|x\|=1\}.
\]
It is always a convex subset of $\CC$. Furthermore, if $D(A)=H$ and $A$ is a bounded operator,
then the spectrum $\sigma(A)$ of $A$ is contained in the closure of $W(A)$. For  details, see 
\cite[pp.264--265]{Da07}.

To state our results, we need to recall one more definition.
If $\Omega$ is a non-empty convex subset of $\CC$, then
its \emph{support function} $h_\Omega:\CC\to(-\infty,\infty]$
is defined by
\[
h_\Omega(z):=\sup\{\Re(z\overline{w}):w\in\Omega\}.
\]

\begin{theorem}\label{T:abstract}
Let $H$ be a complex Hilbert space, let $T$ be an injective bounded  operator on $H$,
and let $\lambda,\mu\in\CC$. Suppose  that $(I+\mu T)$ is invertible.
Then 
\[
\|(I+\lambda T)(I+\mu T)^{-1}\|\le1
\quad\iff\quad
2h_{W(T^{-1})}(\lambda-\mu)\le|\mu|^2-|\lambda|^2.
\]
\end{theorem}

We single out two special cases of this result that are worthy of note.
In the first of these, we write 
$[-\overline{\mu},\mu)$ for $\{-(1-t)\overline{\mu}+t\mu:t\in[0,1)\}$.

\begin{theorem}\label{T:halfplane}
Let $H$ be a complex Hilbert space, and let $T$ be an injective bounded operator on $H$
such that
$W(T^{-1})=\{z\in\CC:\Re z\ge0\}$.
Then, given distinct $\lambda,\mu\in\CC$,
the following statements are equivalent:
\begin{enumerate}[\normalfont(i)]
\item $(I+\mu T)$ is invertible and $\|(I+\lambda T)(I+\mu T)^{-1}\|=1$;
\item $\Re\mu>0$ and $\lambda\in[-\overline{\mu},\mu)$.
\end{enumerate}
\end{theorem}

\begin{theorem}\label{T:fullplane}
Let $H$ be a complex Hilbert space, and let $T$ be an injective bounded operator on $H$
such that $W(T^{-1})=\CC$.
If $\lambda,\mu\in\CC$ with $\lambda\ne\mu$ and if $(I+\mu T)$ is invertible, then 
$\|(I+\lambda T)(I+\mu T)^{-1}\|>1$.
\end{theorem}

Theorems~\ref{T:abstract}, \ref{T:halfplane} and \ref{T:fullplane} will all be proved in \S\ref{S:abstract}.

We now illustrate them using the  Volterra operator
$V:L^2[0,1]\to L^2[0,1]$,
defined by
\[
Vf(x):=\int_0^ xf(t)\,dt \quad(f\in L^2[0,1]).
\]
For background information on $V$,
see e.g.\ \cite[Chapter~7]{GMR23}.
Note, in particular, that $V$ is an injective bounded  operator
whose spectrum is equal to $\{0\}$,
so the abstract results above do indeed apply to $T=V$, and indeed to $T=V^n$
for all $n\ge1$. 

We first need to
determine the numerical range of $V^{-n}$. This is the subject of
the following theorems.

\begin{theorem}\label{T:inverse}
$W(V^{-1})=\{z\in\CC:\Re z\ge0\}$.
\end{theorem}

\begin{theorem}\label{T:higherpowers}
If $n\ge2$, then $W(V^{-n})=\CC$.
\end{theorem}

These results will be proved in \S\ref{S:negpowers}.
Combining Theorems~\ref{T:halfplane} and \ref{T:inverse},
we immediately recover the following result,
originally due to Khadkhuu, Zem\'anek and the second author \cite[Theorem~2.1]{KTZ15}.

\begin{corollary}
Let $\lambda,\mu\in\CC$ with $\lambda\ne\mu$. Then $\|(I+\lambda V)(I+\mu V)^{-1}\|=1$
if and only if $\Re \mu>0$ and  $\lambda\in[-\overline{\mu},\mu)$.
\end{corollary}

Also, Theorems~\ref{T:fullplane} and \ref{T:higherpowers}  together immediately
lead to the following result.

\begin{corollary}\label{C:higherpowers}
If $n\ge2$, then
 $\|(I+\lambda V^n)(I+\mu V^n)^{-1}\|>1$ for all $\lambda,\mu\in\CC$ with $\lambda\ne\mu$.
\end{corollary}

We conclude by remarking that Corollary~\ref{C:higherpowers} 
could also be deduced from a result of ter Elst and Zem\'anek \cite[Proposition~2.3]{tZ18}.
The argument is quite different, and uses  properties specific to the Volterra operaor.


\section{Proofs of Theorems~\ref{T:abstract}, \ref{T:halfplane} and \ref{T:fullplane}}\label{S:abstract}

\begin{proof}[Proof of Theorem~\ref{T:abstract}]
We have the following chain of equivalences:
\begin{align*}
&\|(I+\lambda T)(I+\mu T)^{-1}\|\le1\\
&\iff \|(I+\lambda T)x\|^2\le \|(I+\mu T)x\|^2 \quad(\forall x\in H)\\
&\iff 2\Re \Bigl((\lambda-\mu)\langle Tx,x\rangle\Bigr)\le (|\mu|^2-|\lambda|^2)\|Tx\|^2 \quad(\forall x\in H)\\
&\iff 2\Re \Bigl((\lambda-\mu)\langle y,T^{-1}y\rangle\Bigr)\le (|\mu|^2-|\lambda|^2)\|y\|^2 \quad(\forall y\in T(H))\\
&\iff 2\Re \Bigl((\lambda-\mu)\overline{w}\Bigr)\le |\mu|^2-|\lambda|^2\quad(\forall w\in W(T^{-1}))\\
&\iff 2h_{W(T^{-1})}(\lambda-\mu)\le|\mu|^2-|\lambda|^2.
\end{align*}
This completes the proof.
\end{proof}

\begin{proof}[Proof of Theorem~\ref{T:halfplane}]
Fix $\lambda,\mu\in\CC$, 
and, for the time being, suppose also that $(I+\mu T)$ is invertible.

As $W(T^{-1})$ is unbounded, we  have $0\in\sigma(T)$, so $1\in \sigma((I+\lambda T)(I+\mu T)^{-1})$,
and consequently $\|(I+\lambda T)(I+\mu T)^{-1}\|\ge1$. 

By assumption, $W(T^{-1})=\{z:\Re z\ge0\}$, so
\[
h_{W(T^{-1})}(z)=
\begin{cases}
0, &\text{if~} z\in(-\infty,0],\\
\infty, &\text{otherwise}.
\end{cases}
\]
Using Theorem~\ref{T:abstract}, it follows that
\[
\|(I+\lambda T)(I+\mu T)^{-1}\|\le 1
\iff \lambda-\mu\in(-\infty,0] \text{~and~}|\mu|\ge|\lambda|\\
\]

It is easy to see that, if $\lambda,\mu$ are distinct, then
\[
\Bigl(\lambda-\mu\in[-\infty,0) \text{~and~}|\mu|\ge|\lambda|\Bigr)
\iff \Bigl(\Re \mu>0 \text{~and~} \lambda\in[-\mu,\mu)\Bigr).
\]

Combining all these remarks, we deduce that, if $\lambda,\mu$ are distinct complex numbers and
$(I+\mu T)$ is invertible, then
\[
\|(I+\lambda T)(I+\mu T)^{-1}\|=1\iff \Re \mu>0 \text{~and~} \lambda\in[-\mu,\mu).
\]

Finally, we note that, although the above equivalence 
was established under the assumption that $(I+\mu T)$ is
invertible, in fact the invertibility of $(I+\mu T)$ is automatic if $\Re\mu>0$.
Indeed, since $W(T^{-1})=\{z:\Re z\ge0\}$,
it follows that $W(T)\subset\{z:\Re z\ge0\}$, whence  
$\sigma(T)\subset\{z:\Re z\ge0\}$
and so $(I+\mu T)$ is invertible. This completes the proof.
\end{proof}

\begin{proof}[Proof of Theorem~\ref{T:fullplane}]
We have $h_{W(T^{-1})}(z)=h_\CC(z)=\infty$ for all $z\in\CC\setminus\{0\}$.
By Theorem~\ref{T:abstract}, it follows that $\|(I+\lambda T)(I+\mu T)^{-1}\|>1$
whenever $(I+\mu T)$ is invertible and $\lambda\ne\mu$.
\end{proof}


\section{Proofs of Theorems~\ref{T:inverse} and \ref{T:higherpowers}}\label{S:negpowers}

\begin{proof}[Proof of Theorem~\ref{T:inverse}]
Let $g\in V(L^2[0,1])$, say $g=Vf$. Then
\[
\Re\langle V^{-1}g,g\rangle=\Re\langle f,Vf\rangle=\frac{1}{2}\langle f,(V+V^*)f\rangle=\frac{1}{2}\Bigl| \int_0^1 f\Bigr|^2\ge0.
\]
Thus $W(V^{-1})\subset\{z:\Re z\ge0\}$.

For the reverse inclusion, we first note that a sufficient condition for $g$ to belong to $V(L^2[0,1])$
is that $g\in C^1[0,1]$ with $g(0)=0$. In this case we have $V^{-1}g=g'$.

In particular, if $n\ge1$ and $g_n(x):=e^{\pm2\pi inx}-1$,
then $g_n\in V(L^2[0,1])$, with $\|g_n\|_2^2=2$ and
\[
\langle V^{-1}g_n,g_n\rangle=\langle g_n',g_n\rangle=\langle \pm2\pi in e^{\pm2\pi inx},\,e^{\pm2\pi inx}-1\rangle=\pm2\pi in.
\]
It follows that $\pm i\pi n\in W(V^{-1})$. 

Likewise, if $n\ge1$ and $h_n(x):=x^n$,
then $h_n\in V(L^2[0,1])$ with $\|h_n\|_2^2=1/(2n+1)$ and
\[
\langle V^{-1}h_n,h_n\rangle=\langle h_n',h_n\rangle=\langle nx^{n-1},x^n\rangle=1/2.
\]
It follows that $n+\frac{1}{2}\in W(V^{-1})$.

Finally, since $\{\pm\pi in,\,(n+\frac{1}{2}):n\ge1\}\subset  W(V^{-1})$ 
and $W(V^{-1})$ is convex, we deduce that $\{z:\Re z\ge0\}\subset W(V^{-1})$.
\end{proof}

\begin{proof}[Proof of Theorem~\ref{T:higherpowers}]
Fix $n\ge2$.
A sufficient condition for $g$ to belong to $V^n(L^2[0,1])$
is that $g\in C^n[0,1]$ with $g^{(k)}(0)=0$ for $0\le k\le n-1$.
In this case $V^{-n}g=g^{(n)}$.

Fix $\theta\in(-\pi/2,\pi/2)$,
and, for $r>0$, set $g_r(x):=x^n\exp(re^{i\theta} x)$.
Then $g_r\in V^n(L^2[0,1])$, and by Leibniz's formula we have
\[
g_r^{(n)}(x)=\sum_{k=0}^n\binom{n}{k}\frac{n!}{k!}x^k(re^{i\theta})^k\exp(re^{i\theta}x).
\]
Hence
\begin{align*}
\langle V^{-n}g_r,g_r\rangle
&=\sum_{k=0}^n\binom{n}{k}\frac{n!}{k!}(re^{i\theta})^k\Bigl\langle x^k\exp(re^{i\theta}x),x^n\exp(re^{i\theta}x)\Bigr\rangle\\
&=\sum_{k=0}^n\binom{n}{k}\frac{n!}{k!}(re^{i\theta})^k\int_0^1 x^{n+k}\exp((2r\cos\theta) x)\,dx.
\end{align*}
Now, for each $m\ge0$, we have
\[
\int_0^1 x^{m}\exp((2r\cos\theta)x\,dx)\sim\int_0^1 \exp((2r\cos\theta))x\,dx
\quad(r\to\infty).
\]
It follows that, as $r\to\infty$,
\begin{align*}
\int_0^1 x^{n+k}\exp((2r\cos\theta)x\,dx)&=\int_0^1 x^{2n}\exp((2r\cos\theta))x\,dx\Bigl(1+o(1)\Bigl)\\
&=\|g_r\|_2^2(1+o(1)).
\end{align*}
Therefore
\[
\langle V^{-n}g_r,g_r\rangle=\Bigl(r^ne^{in\theta}+O(r^{n-1})\Bigr)\|g_r\|_2^2(1+o(1))\quad(r\to\infty).
\]
Thus,
for each $\theta\in(-\pi/2,\pi/2)$,
the set  $W(V^{-n})$ contains complex numbers of the form $r^ne^{in\theta}(1+o(1))$
as $r\to\infty$.
As $n\ge2$ and $W(V^{-n})$ is convex, we deduce that $W(V^{-n})=\CC$, as claimed.
\end{proof}


\begin{remark}
The  simple descriptions 
of the numerical ranges of negative powers of $V$ in Theorems~\ref{T:inverse}
and \ref{T:higherpowers} are in marked contrast with those of the positive powers of~$V$,
which are more complex.
It has long been known that $W(V)$ is the closed region bounded 
by the vertical segment $[-i/2\pi,\,i/2\pi]$ and the curves
\[
t\mapsto \Bigl(\frac{1-\cos t}{t^2}\Bigr) \pm i\Bigl(\frac{t-\sin t}{t^2}\Bigr) \quad(t\in[0,2\pi]).
\]
This was stated without proof in \cite[p.113]{Ha82}. 
Proofs can be found, for example, in \cite{KT18} and \cite{RW23}.
More recently, in \cite{KT20}, Khadkhuu and the second author obtained 
an analogous description of $W(V^2)$.
It is considerably more complicated.
To the best of our knowledge, the precise identification of $W(V^n)$ remains
an open problem for $n\ge3$.

The operator norms of $V^n$
have been intensively studied by many authors
(see e.g.\ \cite[p.1058]{BD09}).
Analytic expressions for $\|V^n\|$ are known for $n=1,2,3$.
For higher values of $n$, exact calculations become too complicated,
and one has to be satisfied with estimates. We suspect that the same may 
turn out to be true 
of $W(V^n)$.
\end{remark}

\section*{Acknowledgements}
We are grateful to the anonymous referee for some very helpful comments.

\section*{Declarations}

\subsection*{\textbf{Funding statement}}
Ransford's research was supported by a grant from the Natural Sciences and Engineering Research Council of Canada. Tsedenbayar has no relevant financial or non-financial interests to disclose.


\bibliographystyle{amsplain}
\bibliography{biblist.bib}

\end{document}